\documentclass[12pt]{amsart}

\usepackage{amssymb}
\usepackage{amsmath}
\usepackage{amsfonts}
\usepackage{graphicx}
\usepackage{amscd}
\usepackage{xcolor}
\usepackage{float}
\usepackage{mathtools}

\newtheorem{thm}{Theorem}[section]
\newtheorem{prop}{Proposition}[section]

\newtheorem{lem}{Lemma}[section]
\newtheorem{cor}{Corollary}[section]

\def\R{\mathbb{R} }

\def\to{\rightarrow}

\def\phi{{\varphi} }

\begin{document}

\title{On the Eigenvalues of the $p\&q-$ Fractional Laplacian}

\author{Sabri Bahrouni}
\address[S. Bahrouni]{ Mathematics Department, Faculty of Sciences, University of Monastir, 5019 Monastir, Tunisia}
\email{sabribahrouni@gmail.com}

\author{Hichem Hajaiej}
\address[H. Hajaiej]{ Department of Mathematics, California State University at Los Angeles,
  Los Angeles, CA 90032, USA}
\email{hhajaie@calstatela.edu}

\author{Linjie Song}
\address[L. Song]{Institute of Mathematics, AMSS, Academia Sinica,
Beijing 100190, China}
\address[L. Song]{University of Chinese Academy of Science,
Beijing 100049, China}
\email{songlinjie18@mails.ucas.edu.cn}

\keywords{ fractional $p\&q-$Laplacien; eigenvalues; continuous spectrum; stability of eigenvalues.\\
\hspace*{.3cm} {\it 2010 Mathematics Subject Classifications}: 35P30, 47J10.
}

\begin{abstract}
We consider the eigenvalue problem for the fractional $p \& q-$Laplacian
\begin{equation}
  \left\{\begin{aligned}
(- \Delta)_p^{s}\, u + \mu(- \Delta)_q^{s}\, u+ |u|^{p-2}u+\mu|u|^{q-2}u=\lambda\ V(x)|u|^{p-2}u\quad & \text{in } \Omega\\
u=0\quad& \text{in}\quad\R^N\backslash\Omega,
\end{aligned}\right.
\end{equation}
where $\Omega$ is an open bounded, and possibly disconnected domain, $\lambda\in\R$, $1<q<p<\frac{N}{s}$, $\mu>0$ with a weight function in $L^\infty(\Omega)$ that is allowed no change sign. We show that the problem has a continuous spectrum.
Moreover, our result reveals a discontinuity property for the spectrum  as the parameter $\mu\to 0^+.$ In addition, a stability property of eigenvalues as $s\to 1^-$ is established.
\end{abstract}

\maketitle

\section{Introduction}

In this paper  we study the following eigenvalue problem for the Dirichlet fractional $p \& q-$Laplacian
\begin{equation}\label{mainequation}
  \left\{\begin{aligned}
(- \Delta)_p^{s}\, u + \mu(- \Delta)_q^{s}\, u+ |u|^{p-2}u+\mu|u|^{q-2}u=\lambda\ V(x)|u|^{p-2}u\quad & \text{in } \Omega\\
u=0\quad& \text{in}\quad\R^N\backslash\Omega,
\end{aligned}\right.
\end{equation}
where $\Omega\subset\R^N$ is a bounded domain, $\lambda\in\R$, $1<q<p<\frac{N}{s}$, $\mu>0$ and $V\in L^{\infty}(\Omega)$ is such that the Lebesgue measure of $\{x\in\Omega\colon\ V(x)>0\}$ is positive. The operator $(-\Delta)_\alpha^s$, with $\alpha \in\{p, q\}$, is the so called fractional $\alpha$-Laplacian operator which, up to normalization factors, may be defined for every function $u \in \mathcal{C}_c^{\infty}\left(\mathbb{R}^N\right)$ as
$$
(-\Delta)_\alpha^s u(x)=2 \lim _{r \rightarrow 0} \int_{\mathbb{R}^N \backslash \mathcal{B}_r(x)} \frac{|u(x)-u(y)|^{\alpha-2}(u(x)-u(y))}{|x-y|^{N+s \alpha}} d y \quad\left(x \in \mathbb{R}^N\right) .
$$
Problems of this type appear in the case of two different materials that involve power hardening exponents $p$ and $q$. Let us point out that the combination of the operators $(- \Delta)_p^{s}\, u$ and $(- \Delta)_q^{s}$ appearing in problem \eqref{mainequation} can be seen as a class of fractional unbalanced double phase problems. Such patterns are strictly connected with mathematical physics (fractional white-noise limit, fractional quantum mechanics, fractional superdiffusion, etc), see \cite{AR}. This operator also can be regarded as the fractional analog of the $p \& q-$Laplacian $-\Delta_p-\Delta_q$, which comes from the study of general reaction diffusion equations with nonhomogeneous diffusion
$$
u_t=\operatorname{div}[A(u) \nabla u]+c(x, u) \text { and } A(u)=|\nabla u|^{p-2}+|\nabla u|^{q-2} .
$$
These kinds of problems stem from a wide range of important applications, including models of elementary particles, plasma physics, biophysics,  reaction-diffusion equations, elasticity theory, etc, see the survey paper \cite{Mosconi} for a more detailed account. Several results for $p \& q-$Laplacian problems obtained in bounded domains and in $\mathbb{R}^N$ can be found in \cite{BCS, BB1, BB2, CD, Figueiredo, Tanaka1} and the references therein. Due to the numerous applications and challenging problems due to the nature of the fractional $p\& q-$
operators, the corresponding non-local problems have attracted the attention of many researchers in recent years. For instance, see \cite{AAI, Ambrosio, AR, AI, BM} for the existence and multiplicity results for problems involving the fractional $p\& q-$Laplacian. However, to the best of our knowledge, there is not much literature available regarding eigenvalues problems including the fractional $p\& q-$Laplacian. The latter is crucial to understand and to derive many critical properties of the operator.\\

The study of eigenvalue problems has been a challenging labor whose beginning dates back to the mid-20th century, more precisely: Given two functionals $\mathcal{A}$ and $\mathcal{B}$ defined on a suitable space $\mathcal{X}$ and a prescribed number $c$, the task of analyzing the existence of numbers $\lambda\in\R$ and elements $u\in \mathcal{X}$ satisfying (in some appropriated sense) equations of the type
$$
\mathcal{B}'(u)=\lambda \mathcal{A}'(u), \qquad \mathcal{A}(u)=c,
$$
where $\mathcal{A}'$ and $\mathcal{B}'$ denote the Fr\'echet derivatives of the functionals.

The prototypical $p-$Laplace operator ($p>1$) then became a focus of study, and in particular, to understand its spectral structure:  Given an open and bounded set $\Omega\subset \R^N$,  to determine the existence  couples $(\lambda,u)$ satisfying the equation
\begin{equation} \label{p.lap}
-div(|\nabla u|^{p-2}u)=\lambda |u|^{p-2}u \text{ in } \Omega, \quad u=0 \text{ on } \partial \Omega
\end{equation}
in a suitable sense. In the seminal work  of Garc\'{\i}a Azorero and Peral Alonso \cite{GP}, they proved the existence of a variational sequence of eigenvalues tending to $+\infty$. However, they couldn't establish whether it exhausts the spectrum unless $p=2$ or $N=1$. Several properties on eigenvalues (and their corresponding eigenfunctions) were addressed by Anane et al \cite{A87, A94} and Lindqvist \cite{L90}, among others, and also for more general boundary conditions than  Dirichlet.

The non-local counterpart of problem \eqref{p.lap} governed by the well-known \emph{fractional $p-$Laplace} operator  takes the form
\begin{equation} \label{frac.p.lap}
(-\Delta_p)^s u =\lambda |u|^{p-2}u \text{ in } \Omega, \quad u=0 \text{ in } \R^N \setminus \Omega.
\end{equation}
The main difficulty here comes from the fact that this operator takes into account long-range  interactions. Problem \eqref{frac.p.lap} was introduced in \cite{L-L}. Several properties on eigenvalues and eigenfunctions were obtained in \cite{DPS15, FP14, L-L}.  Allowing a growth behavior more general than power in problem \eqref{frac.p.lap} existence of eigenvalues was studied in \cite{sabri}. See also \cite{BGH, CGH, TH} for problems involving the mixed fractional Laplacians.\\

Letting $\mu\to 0^+$, equation \eqref{mainequation} becomes

\begin{equation}\label{quation-p}
  \left\{\begin{aligned}
(- \Delta)_p^{s}\, u + |u|^{p-2}u=\lambda\ V(x)|u|^{p-2}u\quad & \text{in } \Omega\\
u=0\quad& \text{in}\quad\R^N\backslash\Omega,
\end{aligned}\right.
\end{equation}

We define
\begin{align}\label{lambda1}
  \lambda_1(p,V)=\inf_{u\in W^{s,p}_0 (\Omega),\atop u\neq0}\left\{ \frac{[u]_{s,p}^p+\|u\|_p^p}{\int_{\Omega}V(x)|u|^{p}d x}\colon\quad  \int_{\Omega}V(x)|u|^{p}d x > 0\right\},
\end{align}
where $[\cdot]_{s,p}$ is the Gagliardo semi-norm (see Section 2 for the precise definition).

By Theorem 3.3 of \cite{Bonder} for the special linear case of the fractional Laplacian and $V\equiv1$, the infimum in \eqref{lambda1} is achieved and $\lambda_1(p,1)>0$.

We aim to give a complete study of the spectrum of the previous equation and we plan to prove the following statements:

\begin{enumerate}
  \item [(i)] There exists a nondecreasing sequence of nonnegative eigenvalues obtained by the Alexander-Spanier Cohmology with $\mathbb{Z}_2-$coefficients method $(\lambda_k)$ tending to $\infty$ as $k\to\infty$.
  \item [(ii)] The first eigenvalue $\lambda_1(p,V)$ is simple and only eigenfunctions associated with $\lambda_1(p,V)$ do not
change sign.
  \item [(iii)] The first eigenvalue $\lambda_1(p,V)$ is isolated.
  \item [(iv)] We characterize the second eigenvalue by means of variational formulations, where the second eigenvalue $\lambda_2(p,V)$ is well-defined, i.e.,
  $$
  \lambda_2(p,V)=\min\{\lambda\colon\ \lambda\ \text{is an eigenvalue of \eqref{quation-p} and}\ \lambda>\lambda_1(p,V)\} > \lambda_1(p,V).
  $$
\end{enumerate}

Because of the non-homogeneity of the operator in problem \eqref{mainequation}, many of the properties that eigenvalues of problem \eqref{quation-p} are not inherited in the non-homogeneous case: for instance, isolation, simplicity and a variational characterization of the first eigenvalue, or a variational formula for the second one. Moreover,  the spectrum of \eqref{mainequation} will be continuous, and in principle, it is not clear the meaning of a first or second eigenvalue. In other word if we denote by $\sigma^s_{p,q}(V)$ the spectrum of \eqref{mainequation}, which main
$$
\sigma^s_{p,q}(V)=\left\{\lambda\in\R\colon\ \lambda\ \text{is an eigenvalue of}\ \eqref{mainequation}\right\},
$$
we can completely describe the spectrum as follows:
$$
\sigma^s_{p,q}(V)=\begin{cases}
                    (\lambda_1(p,V),+\infty), & \mbox{if } V\geq0\\
                    (-\infty,-\lambda_1(p,-V))\cup(\lambda_1(p,V),+\infty), & \mbox{otherwise}.
                  \end{cases}
$$
The main result reads:

\begin{thm}\label{nonexistence}
$~$
  \begin{enumerate}
    \item[$(i)$] If $-\lambda_{1}\left(p,-V\right)\leq\lambda\leq \lambda_{1}\left(p,V\right)$ holds, then for any $\mu>0$, \eqref{mainequation} has no non-trivial solutions.
    \item[$(ii)$] If $\lambda_{1}\left(p,V\right)\leq\lambda\leq \lambda_{2}\left(p,V\right)$ holds,  then for any $\mu>0$, \eqref{mainequation} has no sign-changing solutions.
  \item[$(iii)$] Tf $\lambda>\lambda_{1}\left(p,V\right)$ or $\lambda <-\lambda_{1}\left(p,-V\right)$ holds, then for any $\mu>0$, \eqref{mainequation} has at least one positive solution.
      \end{enumerate}
\end{thm}

Our second aim, is to analyze the stability of eigenvalues for equation \eqref{mainequation} when the fractional parameter goes to $1^-$, in which case the limiting problem of \eqref{mainequation} is formally given by
\begin{equation}\label{mainequation2}
  \left\{\begin{aligned}
(- \Delta)_p\, u + \mu(- \Delta)_q\, u+ |u|^{p-2}u+\mu|u|^{q-2}u=\lambda\ V(x)|u|^{p-2}u\quad & \text{in } \Omega\\
u=0\quad& \text{on}\quad\partial\Omega.
\end{aligned}\right.
\end{equation}

More precisely, we aim to show that any accumulation point $\lambda$ of the set $\left\{\lambda^{s}\right\}_{s \in(0,1)}$ (different from $\lambda_1(p,V)$) belongs to $\sigma_{p,q}^{1}$, where $\lambda^{s} \in \sigma_{p,q}^{s}$ is such that $\sup _{s \in(0,1)} \lambda^{s}<\infty$ (see Theorem \ref{thm5.2}). For the limit problem \eqref{mainequation2}, the continuity of the spectrum has been first studied by \cite{Tanaka}. This investigation pursued in the second part of the paper is based upon the results by Bourgain, Brezis, and Mironescu \cite{Brezis, BBM}. We have that if $u \in W_0^{1, p}(\Omega)$
$$
\lim _{s \nearrow 1}(1-s)[u]_{s,p,\mathbb{R}^N}^p=K(p, N)\|\nabla u\|_{p,\Omega}^p,
$$
where the constant $K(p,N)$ is given by
$$
K(p, N):=\frac{1}{p} \int_{\mathbb{S}^{N-1}}|\langle\sigma, \mathbf{e}\rangle|^p d \mathcal{H}^{N-1}(\sigma), \quad \mathbf{e} \in \mathbb{S}^{N-1}.
$$
This line of research was developed, for instance, in \cite{Biccari, Secchi, BPS16, Bonder-Dussel, Bonder-Salort, Bonder-S-S}, and was found to be applicable to many cases.\\

The rest of this paper is organized as follows. Section 2 includes preliminaries and useful results about fractional order Sobolev spaces, the fractional $p-$Laplacian, and the fractional $p \& q-$Laplacian. In Section 3, we give a complete study of the spectrum when $\mu = 0$. Then we study the case when $\mu > 0$ in Section 4 and the case when $s \rightarrow 1^-$ in Section 5.

\section{Preliminaries and useful results}

\subsection{Fractional order Sobolev spaces}
Given a fractional parameter $0<s<1$ and $1 \leq p<\infty$, we define the fractional order Sobolev space $W^{s, p}\left(\mathbb{R}^{N}\right)$ as
$$
W^{s, p}\left(\mathbb{R}^{N}\right):=\left\{u \in L^{p}\left(\mathbb{R}^{N}\right):[u]_{s, p}<\infty\right\}
$$

where the Gagliardo semi-norm $[\cdot]_{s, p}$ is defined as
$$
[u]_{s, p}^p:= K(N,s,p)\iint_{\mathbb{R}^{N} \times \mathbb{R}^{N}} \frac{|u(x)-u(y)|^{p}}{|x-y|^{N+s p}}\, d x d y.
$$
The constant $K(N,s,p)$ is a normalizing constant defined as
$$
K(N,s,p):=(1-s)\mathcal{K}(N,p), \qquad \text{ where }\quad
\mathcal{K}(N,p)^{-1}:=\frac1p \int_{\mathbb{S}^{N-1}} \omega_n^p\,dS_\omega
$$
and $\mathbb{S}^{N-1}$ is the unit sphere in $\R^N$.

The main feature of this constant is that, for any $u\in L^p(\R^N)$, it holds that
$$
\lim_{s\uparrow 1} [u]_{s,p}^p = \|\nabla u\|_p^p,
$$
where the limit above is understood as an equality if $u\in W^{1,p}(\R^N)$ and $\liminf_{s\uparrow 1}[u]_{s,p}^p =\infty$ otherwise. This is the celebrated result obtained by Bourgain, Brezis and Mironescu  in \cite{Brezis}.

The  space $W^{s,p}(\R^N)$ is a Banach space with norm given by $\|u\|_{s, p}=\left(\|u\|_{p}^{p}+[u]_{s, p}^{p}\right)^{\frac{1}{p}}$. This is a separable space and reflexive for $p>1$.
Given an open set $\Omega \subset \mathbb{R}^{N}$, we then consider
$$
W_{0}^{s, p}(\Omega):=\left\{u \in W^{s, p}\left(\mathbb{R}^{N}\right): u=0 \text { a.e. in } \mathbb{R}^{N} \backslash \Omega\right\}.
$$

By Poincar\'e inequality, it follows that $[\cdot]_{s, p}$ defines a norm in $W_{0}^{s, p}(\Omega)$ equivalent to $\|\cdot\|_{s, p}.$

It is well-known that $W_0^{s,p}(\Omega)\hookrightarrow L^r(\Omega)$ with
$$
\begin{array}{ll}
r\in(1, p^*_s) &\quad  \text{ if } sp<N\\
r=p_s^* &\quad \text{ if } sp\neq N\\
r\in (1,\infty) &\quad \text{ if } sp= N
\end{array}
$$

We also consider the usual $W_0^{1,p}(\Omega)$ endowed with the norm $[u]_{1,p}^p=\int_{\Omega}|\nabla u|^p dx$.

\subsection{The fractional $p-$Laplacian}
For any $0<s\leq 1\leq p <\infty$ we consider the energy functional
\begin{align*}
\mathcal{J}_{s,p}\colon W^{s,p}_0(\Omega)\to \R, \qquad \mathcal{J}_{s,p}(u)=
\begin{cases}
\frac{1}{p}[u]_{s,p}^p &\text{ if } 0<s<1\\
\frac{1}{p}\|\nabla u\|_p^p &\text{ if } s=1.
\end{cases}
\end{align*}
We define a nonlinear operator $A_s\colon W^{s,p}_0(\Omega) \to W^{-s,p'}(\Omega)$
\begin{align*}
\left\langle A_s(u), v\right\rangle &= \frac{K(N,s,p)}{2}\iint_{\mathbb{R}^{2 N}} \frac{|u(x)-u(y)|^{p-2}(u(x)-u(y))(v(x)-v(y))}{|x-y|^{N+s p}} \, d x d y \\
\left\langle A_1(u), v\right\rangle &= \int_\Omega |\nabla u|^{p-2}u \cdot \nabla v \,dx.
\end{align*}
The functional $\mathcal{J}_{s,p}$ is Fr\'echet differentiable and $\mathcal{J}'_{s,p}\colon W^{s,p}_0(\Omega) \to W^{-s,p'}(\Omega)$ is continuous and for every $u,v\in W^{s,p}_0(\Omega)$
\begin{align*}
\left\langle\mathcal{J}_{s,p}'(u), v\right\rangle &= \left\langle A_s(u), v\right\rangle\\
\left\langle\mathcal{J}_{1,p}'(u), v\right\rangle &= \left\langle A_1(u), v\right\rangle .
\end{align*}

Therefore, for any $0<s\leq 1<p$ we define the \emph{fractional $p-$Laplace operator} as $(-\Delta_p)^s:=\mathcal{J}'_{s,p}$.

\medskip

We recall a useful property of the fractional $p-$Laplacian. Since $W^{s,p}_0(\Omega)$ is uniformly convex, it follows from \cite[Proposition 1.3]{MR2640827} that the operator $(-\Delta_p)^s$ is of type $(S)$, i.e., every sequence $(u_j) \subset W^{s,p}_0(\Omega)$ such that $u_j \to u$ weakly and $(\mathcal{J}'_{s,p}(u_j),u_j - u) \to 0$ has a subsequence that converges strongly to $u$.

\medskip

The usual $p-$Laplace operator can be seen as the limit of the fractional $p-$Laplacian as $s\uparrow 1$ in the following sense:
\begin{lem}\cite[Lemma 2.8]{Bonder-Salort}\label{cvpsprel}
  Let $s_{k} \uparrow 1$ and $v_{k} \in W_0^{s,p}(\Omega)$ be such that $\sup _{k \in \mathbb{N}}\left[v_{k}\right]_{s_{k}, p}^{p}<\infty$. Assume, without loss of generality, that $v_{k} \rightarrow v$ strongly in $L^{p}(\Omega) .$ Then, for every $u \in W_{0}^{1, p}(\Omega)$, we have
$$
\left\langle\left(-\Delta_{p}\right)^{s_{k}} u, v_{k}\right\rangle \rightarrow\left\langle-\Delta_{p} u, v\right\rangle.
$$
\end{lem}

\subsection{Setting of the problem}
In the case of the usual Sobolev spaces, for any $1\leq p<q \leq \infty$, it is easy to see that $W^{1,q}_0(\Omega) \subset W^{1,p}_0(\Omega)$. In the fractional case this kind of embedding is NOT TRUE. In fact, in \cite[Lemma 2.6]{BLP} it is proved that
$$
W^{s_1,p}_0(\Omega)\hookrightarrow W^{s_2,q}_0(\Omega) \quad \text{for any }0<s_1<s_2<1 \leq q< p<\infty,
$$
this also holds when $p=q$ (see \cite[Theorem 2.2]{Hajaiej-Perera}).
However, the embedding
$$
W^{s,p}_0(\Omega)\hookrightarrow W^{s,q}_0(\Omega) \quad \text{ for any }0<s<1 \leq q< p<\infty.
$$
is not true (see \cite[Theorem 1.1]{Petru}). So, in order to deal with our problem \eqref{mainequation}, we consider the space
$$
\mathcal{W}^s:=W^{s,p}_0(\Omega)\cap W^{s,q}_0(\Omega)
$$
endowed with the norm $[\cdot]_{s}:=[\cdot]_{s, p}+[\cdot]_{s, q}$.

\medskip

We also consider the energy functional related to equation \eqref{mainequation} given by
$$
\mathcal{J}_s(u)=\frac{1}{p}[u]_{s, p}^{p}+\frac{\mu}{q}[u]_{s, q}^{q}.
$$
It is well-defined on the space $\mathcal{W}^s$ for $0<s\leq 1$. In fact, it is straightforward to see that $\mathcal{J}_s\in C^{1}\left(\mathcal{W}^s, \mathbb{R}\right)$ and its Fr\'echet derivative is given by
\begin{align*}
\left\langle \mathcal{J}_s'(u), v\right\rangle = \left\langle\mathcal{J}_{s,p}'(u) , v\right\rangle
+ \mu \left\langle \mathcal{J}_{s,q}'(u), v \right\rangle
\end{align*}
for any $u, v \in \mathcal{W}^s$.

The expression $\mathcal{J}'_{s}$ defines what we call \emph{fractional $p \& q-$Laplacian}.

\medskip

We say that a function $u\in \mathcal{W}^s$ is an eigenfunction of \eqref{mainequation} corresponding to the eigenvalue $\lambda$ if
$$
\left\langle\mathcal{J}_s'(u), v\right\rangle + \int_{\Omega} (|u|^{p-2} u +\mu |u|^{q-2} u)  v\, d x=
\lambda\int_{\Omega}V(x) |u|^{p-2} u v \,d x
\quad \text{ for all }v\in \mathcal{W}^s.
$$

\medskip

\subsection{Some properties of the fractional $p \& q-$Laplacian}

It is well known that the operator $(-\Delta_p)^s$ is strictly monotone. This is a consequence of the well known Simon inequality (see \cite{Simon})
$$
\left(|a|^{p-2} a-|b|^{p-2} b\right) \cdot(a-b) \geq \begin{cases}c|a-b|^p & \text { if } p \geq 2 \\ c \frac{|a-b|^2}{(|a|+|b|)^{2-p}} & \text { if } 1<p<2,\end{cases}
$$
for any $a, b \in \mathbb{R}^N(N \in \mathbb{N})$, where the constant $c$ depends on $p$ and $N$. From the monotonicity of $(-\Delta_p)^s$, it is easy to see that
$$
\langle \mathcal J'_s (u)-\mathcal J'_s (v),u-v \rangle \geq 0 \qquad u,v\in \mathcal{W}^s.
$$

Moreover, from Lemma \ref{cvpsprel} it is immediate to see that $\mathcal J'_s (u) \to \mathcal J'_1 (u)$ as $s\uparrow 1$ in the following sense:
\begin{lem}\label{cvps}
  Let $s_{k} \uparrow 1$ and $v_{k} \in \mathcal{W}^{s_k}$ be such that $\sup _{k \in \mathbb{N}}\left[v_{k}\right]_{s_{k}}<\infty$. Assume, without loss of generality, that $v_{k} \rightarrow v$ strongly in $L^{p}(\Omega) .$ Then, for every $u \in W_{0}^{1, p}(\Omega)$, we have
$$
\langle \mathcal{J}'_{s_k}(u),v_k\rangle \rightarrow \langle \mathcal{J}'_1 (u), v\rangle, \quad\text{for any}\quad\mu>0.
$$
\end{lem}
The previous lemma says that, for any $\mu>0$,
$$
\left\langle\left(-\Delta_{p}\right)^{s_{k}} u+\mu\left(-\Delta_{q}\right)^{s_{k}} u, v_{k}\right\rangle \rightarrow\left\langle-\Delta_{p} u-\mu\Delta_{q} u, v\right\rangle.
$$

\section{The case $\mu\to 0^+$}

\subsection{The first eigenvalue}
In this section we assume that the potential $V$ is a positive weight of class $L^{r}(\Omega)$, with $r>\frac{N}{sp}$. Let
$$
L^{p}(\Omega, V)=\left\{u: \Omega \rightarrow \mathbb{R} \text { measurable }: V|u|^{p} \in L^{1}(\Omega)\right\}.
$$
This space is endowed with the norm $\|\cdot\|_{p, V}$ which is defined by
$$
\|u\|_{p, V}^p=\int_{\Omega} V(x)|u(x)|^{p}\, d x.
$$
We remark that the embedding $W_0^{s, p}(\Omega) \hookrightarrow L^{p}(\Omega, V)$ is compact. Indeed, by our assumption, for any $p>1$, $r'p<p^*$, and then, by \cite[Corollary 7.2]{PNV} $W_0^{s, p}(\Omega) \hookrightarrow L^{r' p}(\Omega)$ is compact,  and $L^{r' p}(\Omega) \hookrightarrow L^{p}(\Omega, V)$ is continuous, since $\|u\|_{p, V}^{p} \leq\|V\|_{r}\|u\|_{r^{\prime} p}^{p}$ for all $u \in L^{r^{\prime} p}(\Omega)$ by the H\"{o}lder inequality.

\begin{prop}\label{infimum}
The infimum in \eqref{lambda1} is attained by a function $u_1\in W^{s,p}_0 (\Omega)$ normalized such that $\int_{\Omega}V(x)|u|^{p}d x=1$.
\end{prop}

\begin{proof}
Observe that we can write
$$
\lambda_1(p,V)=\inf_{u\in \mathcal{S}} \frac{\mathcal{P}(u)}{\mathcal{B}(u)}
$$
with $\mathcal{P}(u)=[u]_{s,p}^p+\|u\|_p^p$, $\mathcal{B}(u)=\|u\|_{p,V}^p$ and $\mathcal{S}=\{ u\in W^{s,p}_0 (\Omega)\backslash\{0\},\ \|u\|_{p,V}=1 \}$. Then, since $\mathfrak{B}$ is continuous with respect to weak sequential convergence, $\mathcal{B}^{'}(u)=0$ only at $u=0$, and $\mathcal{P}$ is coercive on $W^{s,p}_0(\Omega)\cap \{\mathfrak{B}(u)\leq c\}$, we can apply, for instance,  Theorem 6.3.2 of \cite{Berger} to conclude the result.
\end{proof}

\begin{prop}[Minimum principle]\label{stricpositive}
  If $u$ is a nonnegative eigenfunction of \eqref{quation-p} associated to $\lambda$ then $u>0$ in $\Omega$.
\end{prop}

\begin{proof}
  In light of \cite[Theorem A.1]{Brasco-Franzina} it is enough to prove that if $u$ is an eigenfunction associated to $\lambda$ then $u\neq 0$ in all connected components of $\Omega$. Suppose, by contradiction, that there is $K$ a connected component of $\Omega$ such that $u\equiv 0$ in $K$.

  Taking $\psi\in C^\infty_c (K)$ as a test function, we get

  \begin{align*}
     &\iint_{\mathbb{R}^{2 N}} \frac{|u(x)-u(y)|^{p-2}(u(x)-u(y))(\psi(x)-\psi(y))}{|x-y|^{N+s p}} d x d y=0\\
     \Leftrightarrow& \int_{K}\psi(y)\int_{\Omega\backslash K}(u(x))^{p-1}\left(\frac{1}{|x-y|^{N+sp}}\right) dxdy = 0,\ \forall \psi\in C^\infty_c (K).
  \end{align*}

Then $u=0$ in $\Omega$, which is a contradiction.
\end{proof}

We prove next lower bound for the Lebesgue measure of $\Omega_{\pm}$.
\begin{lem}\label{isolated}
 If $u$ is an eigenfunction associated to $\lambda>\lambda_1(p,V)$, then
 $$
 (\lambda C \|V\|_r)^\frac{1}{ \frac{p}{\alpha}-\frac{1}{r'}} \leq |\Omega_\pm|
 $$
where $\alpha \in\left(p r^{\prime}, p_{s}^{*}\right)$ and is a constant independent of $V, \lambda$ and $u$.
\end{lem}

\begin{proof}
Due to Proposition \ref{lambda1p}, $u_\pm$ are not trivial. Observe that $u_{+} \in W^{s, p}(\Omega)$ and
$$
\left|u_{+}(x)-u_{+}(y)\right|^p \leq|u(x)-u(y)|^{p-2}(u(x)-u(y))\left(u_{+}(x)-u_{+}(y)\right)
$$
for all $(x, y) \in \mathbb{R}^{N} \times \mathbb{R}^{N}$. Then, using H\"{o}lder's inequality, we have
\begin{align*}
  [u_{+}]_{s, p}^{p}  &\leq \langle (-\Delta_p)^s u, u_+\rangle=\lambda \int_{\Omega}V(x) u_{+}^{p} \,d x-\int_{\Omega}  u_{+}^{p} \,d x\\
&\leq \lambda\|V\|_{r}\|u_+\|_{p r^{'}}^p \leq \lambda\|V\|_{r}|\Omega_+|^{\frac{1}{r^{'}}-\frac{p}{\alpha}}\|u_+\|_{\alpha}^p.
\end{align*}
On the other hand, by Sobolev embedding, $\|u_{+}\|_{\alpha}^p \leq C\left[u_{+}\right]_{s, p}^p$, which gives
$$
1 \leq \lambda C\|V\|_{r}\left|\Omega_{+}\right|^{\frac{1}{r^{\prime}}-\frac{p}{\alpha}}
$$
and then the result holds for $|\Omega_+|$. The inequality for $|\Omega_-|$ is obtained analogously.
\end{proof}

The following statement summarizes some basic facts about the first eigenvalue $\lambda_1(p,V)$.

\begin{prop}\label{lambda1p}
$~$
  \begin{enumerate}
    \item [(i)]  $\lambda_1(p,V)$ is simple, and all eigenfunctions associated to an eigenvalue $\lambda>\lambda_1(p,V)$ must change sign.
    \item [(ii)]  All eigenfunctions associated to $\lambda_1(p,V)$ have constant sign.
    \item [(iii)] If $\Omega$ is a ball, then all eigenfunctions associated to $\lambda_1(p,V)$ are radially symmetric and radially monotone.
    \item [(iv)] $\lambda_1(p,V)$ is isolated.
  \end{enumerate}
\end{prop}

\begin{proof}
 $(i)$. It follows by Picone's identity.

$(ii)$. The eigenfunction $u_1$ can be chosen to be positive. Indeed, since
$$
||u_1(x)|-|u_1(y)||\leq |u_1(x)-u_1(y)|\ \text{for all}\ x,y\in \R^N,
$$
we get $[|u_1|]_{s,p}^p\leq[u_1]_{s,p}^p\leq\lambda_1(p,V)$.
  This implies that $|u_1|$ is an eigenfunction associated to $\lambda_1(p,V)$, and by Proposition \ref{stricpositive}, we have $|u_1|>0$.

 $(iii)$. Let $u$ be an eigenfunction associated to $\lambda_1$ in the ball $\Omega$ and denote by $u^*$ the Schwarz symmetrization of $u$. By $(ii)$,  we may assume that  $u>0$. By \cite[Theorem 3]{B} it is known that $\|u^*\|_p=\|u\|_p$, $\|u^*\|_{V,p}=\|u\|_{V,p}$ and $[u^*]_{s,p}\leq [u]_{s,p}$. Hence, we have from Proposition \ref{infimum} that
 $$
 \lambda_1(p,V)\leq [u^*]_{s,p}^p+|u^*\|_p^p\leq [u]_{s,p}^p+|u\|_p^p=\lambda_1(p,V).
 $$
 Thus $u^*$ is an eigenfunction associated to $\lambda_1(p,V)$. By $(i)$, $u^*$ and $u$ are proportional and by the inequality above we obtain $u^*=u$.

$(iv)$. By the definition of the eigenvalue, $\lambda_1(p,V)$ is left-isolated. To prove that $\lambda_1(p,V)$ is right-isolated, we argue by contradiction. We assume that there exists a sequence of eigenvalues $\{\lambda_k\}_{k\in\mathbb{N}}$ such that $\lambda_k>\lambda_1(p,V)$ and $\lambda_k\searrow \lambda_1(p,V)$ as $k\to \infty$. Let $u_k$ be an eigenfunction associated to $\lambda_k$, we can assume that $\|u_k\|^p_{s,V}=1$. Then $\{u_k\}_{k\in\mathbb{N}}$ is bounded in $W^{s,p}_0(\Omega)$ and up to a  subsequence $u_k$ converges to $u$ weakly in $W^{s,p}_0(\Omega)$ and strongly in $L^p(\Omega,V)$. Then $\|u\|^p_{p,V}=1$ and
\begin{align*}
    [u]^p_{s,p} + \|u\|_p^p& \leq \liminf_{k\to \infty} ( [u_k]^p_{s,p} + \|u_k\|_p^p)= \lim_{k\to \infty} \lambda_k\int_{\Omega}V(x)|u_k|^p\, d x\\
    &=\lambda_1(p,V)\int_{\Omega}V(x)|u|^p d x.
  \end{align*}
Hence, $u$ is an eigenfunction associated to $\lambda_1(p,V)$. By $(ii)$, we can assume
that $u > 0$. On the other hand, by the Egorov's theorem, for any $\varepsilon > 0$ there exists
$U_\varepsilon\subset\Omega$ such that $|U_\varepsilon| < \varepsilon$ and $u_k\to u$ uniformly in $\Omega\setminus U_\varepsilon$. This contradicts Lemma \ref{isolated} since, for $\alpha \in\left(p r^{\prime}, p_{s}^{*}\right)$, we have
$$
0<\lim_{k\to\infty} (\lambda_k C \|V\|_r)^\frac{1}{ \frac{p}{\alpha}-\frac{1}{r'}} \leq\left|\Omega_{\pm}\right|\leq\lim_{k\to\infty}|\{x\in\R^N\colon\ u_k<0\}|,
$$
getting a contradiction when $k\to\infty$.
\end{proof}

\subsection{Sequence of eigenvalues}

We now construct a suitable sequence of eigenvalues $(\lambda_k)_k$ of \eqref{quation-p}. The construction is based on the Alexander-Spanier Cohmology with $\mathbb{Z}_2-$coefficients method (see Fadell and Rabinowitz \cite{MR0478189}).

In order to neatly present our findings, we first need to give some basic notations. Denote by $\mathcal{F}$ the class of symmetric subsets of $S$, and denote by $i(A)$ the $\mathbb{Z}_2$-cohomological index of a symmetric set $A \subset W^{s,p}_0(\Omega) \setminus \{0\}$. Let
\[
\mathcal{F}_k = \{M \in \mathcal{F} : i(M) \geq k\},
\]

We set for all $u\in W^{s,p}_0(\Omega)$
$$
I(u)=[u]_{s, p}^{p}+\|u\|_p^{p},\quad J(u)=\int_{\Omega}V(x)|u|^p d x,
$$
$$
S_V=\{u\in W^{s,p}_0(\Omega):\ \int_{\Omega}V(x)|u|^p d x=1 \}.
$$
We denote by $\bar{I}$ the restriction of $I$ to $S_V$. For all $\lambda>0$, $\lambda$ is a critical value of $\bar{I}$ if and only if it is an eigenvalue of \eqref{quation-p}.

Now, we can construct the sequence $(\lambda_k)_k$ of eigenvalues for problem \eqref{quation-p}, setting $$\lambda_k=\inf_{A\in \mathcal{F}_k}\sup_{u\in A}I(u).$$

Since $S_V=J^{-1}(1)$, then $u\in S_V$ is a critical point of $\bar{I}$ if and only if there is some $\alpha\in\R$ such that $I^{'}(u)=\alpha J^{'}(u)$.

\begin{prop}\label{PS}
  The functional $\bar{I}$ satisfies the Palais-Smale condition at any level $c\in\R$ and for all $k\in \mathbb{N}$, $\lambda_k$ is an eigenvalue of problem \eqref{quation-p}. Moreover, $\lambda_k\to \infty$.
\end{prop}

\begin{proof}
  Let $c\in\R$ and $(u_n)_n\subset S_V$ be a $(PS)_c$ sequence of $\bar{I}$. Hence there exists a suitable sequence $(\alpha_n)_n\subset\R$ such that
  $$
  \bar{I}(u_n)\to c\quad\text{and}\quad \bar{I}^{'}(u_n)-\alpha_n J^{'}(u_n)\to 0\ \text{in}\ W^{s,p}_0(\Omega).
  $$
  Clearly $(u_n)_n$ is bounded in $W^{s,p}_0(\Omega)$. Therefore we can pass to a subsequence $u_n\rightharpoonup u$ in $W^{s,p}_0(\Omega)$. Hence $u_n\to u$ in $L^p(\Omega,V)$. In particular $u\in S_V$. Now,
  $$
  \bar{I}(u_n)-\alpha_n=(\bar{I}^{'}(u_n), u_n)-\alpha_n(J^{'}(u_n),u_n)\to 0.
  $$
  Hence
  $$
  \alpha_n=\bar{I}(u_n)+o(1)\to c.
  $$
  Notice that, for all $n\in\mathbb{N}$, we have
  \begin{align*}
     |(\bar{I}^{'}(u_n), u_n-u)|& =(A_p u_n,u_n-u)+\int_{\Omega}|u_n|^{p-2}u_n(u_n-u)dx\\
     &=\alpha_n\left(\int_{\Omega}V(x)|u_n|^{p-2}u_n(u_n-u)\right)+o(1)\\
     &\leq \alpha_nC\|\omega_n\|_{V,p^{'}}\|u_n-u\|_p,
  \end{align*}
  and the latter vanishes as $n\to\infty$ because $\|\omega_n\|_{V,p^{'}}\to \|\omega\|_{V,p^{'}}$ where $\omega_n=|u_n|^{p-2}u_n$ and and similarly  $\omega=|u|^{p-2}u$. Hence, by the $(S)$-property of $A_p$, we get $u_n\to u$ in $W^{s,p}_0(\Omega)$.In other words, $\bar{I}$ verifies the $(P S)_c$ condition for all $c\in\R^+$.\\

  An application of \cite[Theorem 4.6]{MR2640827} to the even functional $I$ shows that $\lambda_k\to \infty$ as $k\to\infty$.
\end{proof}

\subsection{The second eigenvalue}

We start by defining
$$
\lambda_{2}(\Omega)=\inf _{f \in C_{1}(\Omega)} \max _{u \in \operatorname{lm} (f)}I(u)
$$
where the set $\mathcal{C}_{1}(\Omega)$ is given by
$$
\mathcal{C}_{1}(\Omega)=\left\{f: \mathbb{S}^{1} \rightarrow \mathcal{S}_{V}: f \text { odd and continuous }\right\}
$$

\begin{thm}
  The quantity $\lambda_{2}(\Omega)$ is an eigenvalue. Moreover we have $\lambda_{1}(p,V)<\lambda_{2}(\Omega)$ and every eigenfunction $u \in S_{V}$ associated to $\lambda_{2}(\Omega)$ has to change sign.
\end{thm}

\begin{proof}
  Let us prove that $ \lambda_{2}(\Omega)$ is an eigenvalue. In order to show that $\lambda_{2}(\Omega)$ is a critical point of the functional $I$ defined on the manifold $S_{V}$, it is sufficient to check that $I$ verifies the Palais-Smale condition. The claim will then follow by Proposition \ref{PS}.

  Now, we prove that $ \lambda_{2}(\Omega)>\lambda_{1}(p,V) .$ Let us argue by contradiction and suppose that
$$
\lambda_{2}(\Omega)=\inf _{f \in \mathcal{C}_{1}(\Omega)} \max _{u \in \operatorname{lm}(f)}I(u)=\lambda_{1}(p,V)
$$

so that, for all $n \in \mathbb{N}$ there exists an odd continuous mapping $f_{n}: \mathbb{S}^{1} \rightarrow S_V$ such that
\begin{equation}\label{pick1}
  \max _{u \in f_{n}\left(S^{1}\right)}I(u) \leq \lambda_{1}(p,V)+\frac{1}{n}.
\end{equation}
Let us denote by $u_{1}$ the global minimizer of $I$ belonging to $\mathcal{S}_V$. From Proposition \ref{lambda1p} $(ii)$, $u_1$ must be nonnegative (or nonpositive). Let $0<\varepsilon \ll 1$ and consider the two neighborhoods
$$
\mathcal{B}_{\varepsilon}^{+}=\left\{u \in \mathcal{S}_V:\left\|u-u_{1}\right\|_{L^{p}(\Omega)}<\varepsilon\right\}
$$
$$
 \mathcal{B}_{\varepsilon}^{-}=\left\{u \in \mathcal{S}_V:\left\|u-\left(-u_{1}\right)\right\|_{L^{p}(\Omega)}<\varepsilon\right\}
$$
which are disjoint, by construction. Since the mapping $f_{n}$ is odd and continuous, for every $n \in \mathbb{N}$ the image $f_{n}\left(\mathbb{S}^{1}\right)$ is symmetric and connected, then it can not be contained in $\mathcal{B}_{\varepsilon}^{+} \cup \mathcal{B}_{\varepsilon}^{-}$, the latter being symmetric and disconnected. So we can pick an element
\begin{equation}\label{pick2}
  u_{n} \in f_{n}\left(S^{1}\right) \backslash\left(\mathcal{B}_{\varepsilon}^{+} \cup \mathcal{B}_{\varepsilon}^{-}\right)
\end{equation}

This yields a sequence $\left\{u_{n}\right\}_{n \in \mathbb{N}} \subset\mathcal{S}_V$, which is bounded in $W_{0}^{s, p}(\Omega)$ by \eqref{pick1}. Hence, there exists a function $v \in S_V$ such that $\left\{u_{n}\right\}_{n \in \mathbb{N}}$ converges to $v$ weakly in $W_{0}^{s, p}(\Omega)$ and strongly in $L^{p}(\Omega)$, possibly by passing to a subsequence. By the weak convergence of $I$ it follows that
$$
I(v) \leq \liminf _{n \rightarrow \infty}I(u_n)=\lambda_{1}(\Omega)
$$
This in turn shows that $v \in \mathcal{S}_V$ is a global minimizer, so that either $v=u_{1}$ or $v=-u_{1}$. On the other hand, by strong $L^{\text {p }}$ convergence we also have
$$
v \in \mathcal{S}_V \backslash\left(B_{\varepsilon}^{+} \cup B_{\varepsilon}^{-}\right) .
$$
This gives a contradiction and thus $\lambda_{2}(\Omega)>\lambda_{1}(p,V)$.

Finally, it follows, directly from the previous assertion and Proposition \ref{lambda1p} $(i)$, that $\lambda_{2}(\Omega)$ admits only sign-changing eigenfunctions.

\end{proof}

\begin{prop}
\label{prop3.5}	For every eigenvalue $\lambda > \lambda_{1}(p,V)$ we have that $\lambda \geq \lambda_{2}(\Omega)$. In particular $\lambda_{1}(p,V)$ is isolated.
\end{prop}

\begin{proof}
Let $\lambda > \lambda_{1}(p,V)$ be an eigenvalue and $u$ be a corresponding eigenfunction. By Proposition \ref{lambda1p} $(i)$, $u$ must change sign. Hence, we assume that $u = u_+ + u_-$ where $u_+ = \max\{u,0\} \neq 0$, $u_- = \min\{u,0\} \neq 0$. By testing the equation solved by $u$ against $u_+$ and $u_-$, we
get
\begin{eqnarray}
	\lambda \int_{\Omega}V(x)|u_+|^pdx = \nonumber \\
     K(N,s,p)\int_{\mathbb{R}^N}\int_{\mathbb{R}^N}\frac{|u(x)-u(y)|^{p-2}(u(x)-u(y))}{|x-y|^{N+ps}}(u_+(x)-u_+(y))dxdy + \int_{\Omega}|u_+|^pdx, \nonumber
\end{eqnarray}
and
\begin{eqnarray}
	\lambda \int_{\Omega}V(x)|u_-|^pdx = \nonumber \\
	K(N,s,p)\int_{\mathbb{R}^N}\int_{\mathbb{R}^N}\frac{|u(x)-u(y)|^{p-2}(u(x)-u(y))}{|x-y|^{N+ps}}(u_-(x)-u_-(y))dxdy + \int_{\Omega}|u_-|^pdx, \nonumber
\end{eqnarray}
which can be rewritten as
$$
\lambda \int_{\Omega}V(x)|u_+|^pdx = K(N,s,p)\int_{\mathbb{R}^N}\int_{\mathbb{R}^N}\frac{|U+W|^{p-2}(U+W)}{|x-y|^{N+ps}}Udxdy + \int_{\Omega}|u_+|^pdx,
$$
and
$$
\lambda \int_{\Omega}V(x)|u_-|^pdx = K(N,s,p)\int_{\mathbb{R}^N}\int_{\mathbb{R}^N}\frac{|U+W|^{p-2}(U+W)}{|x-y|^{N+ps}}Wdxdy + \int_{\Omega}|u_-|^pdx,
$$
where $U(x,y) = u_+(x)-u_+(y)$ and $W(x,y) = u_-(x)-u_-(y)$. Let $(\theta_1, \theta_2) \in \mathbb{S}^1$, by multiplying the previous two identities by $|\theta_1|^p$ and $|\theta_2|^p$ and adding them, we can arrive at
\begin{eqnarray}
	\lambda &=& \frac{K(N,s,p)\int_{\mathbb{R}^N}\int_{\mathbb{R}^N}\frac{[|\theta_1|^p|U+W|^{p-2}(U+W)U + |\theta_2|^p|U+W|^{p-2}(U+W)W}{|x-y|^{N+ps}}dxdy}{|\theta_1|^p\int_{\Omega}V(x)|u_+|^pdx + |\theta_2|^p\int_{\Omega}V(x)|u_-|^pdx} \nonumber \\
	&& + \frac{|\theta_1|^p\int_{\Omega}|u_+|^pdx + |\theta_2|^p\int_{\Omega}|u_-|^pdx}{|\theta_1|^p\int_{\Omega}V(x)|u_+|^pdx + |\theta_2|^p\int_{\Omega}V(x)|u_-|^pdx}. \nonumber
\end{eqnarray}
Now we define the following element of $\mathcal{C}_{1}(\Omega)$
$$
f(\theta) = \frac{\theta_1u_+ + \theta_2u_-}{(|\theta_1|^p\int_{\Omega}V(x)|u_+|^pdx + |\theta_2|^p\int_{\Omega}V(x)|u_-|^pdx)^{1/p}}, \theta = (\theta_1,\theta_2) \in \mathbb{S}^1.
$$
Then
\begin{eqnarray}
	I(f(\theta)) &=& \frac{K(N,s,p)\int_{\mathbb{R}^N}\int_{\mathbb{R}^N}\frac{|\theta_1u_+(x) + \theta_2u_-(x) - (\theta_1u_+(y) + \theta_2u_-(y))|^p}{|x-y|^{N+ps}}dxdy}{|\theta_1|^p\int_{\Omega}V(x)|u_+|^pdx  + |\theta_2|^p\int_{\Omega}V(x)|u_-|^pdx} \nonumber \\
	&& + \frac{\int_{\Omega}|\theta_1u_+ + \theta_2u_-|^pdx}{|\theta_1|^p\int_{\Omega}V(x)|u_+|^pdx + |\theta_2|^p\int_{\Omega}V(x)|u_-|^pdx}. \nonumber \\
	&=&  \frac{K(N,s,p)\int_{\mathbb{R}^N}\int_{\mathbb{R}^N}\frac{|\theta_1U + \theta_2W|^p}{|x-y|^{N+ps}}dxdy}{|\theta_1|^p\int_{\Omega}V(x)|u_+|^pdx  + |\theta_2|^p\int_{\Omega}V(x)|u_-|^pdx} \nonumber \\
	&& + \frac{|\theta_1|^p\int_{\Omega}|u_+|^pdx + |\theta_2|^p\int_{\Omega}|u_-|^pdx}{|\theta_1|^p\int_{\Omega}V(x)|u_+|^pdx + |\theta_2|^p\int_{\Omega}V(x)|u_-|^pdx}. \nonumber
\end{eqnarray}
By (4.7) in the proof of \cite[Proposition 4.2]{BP}, we have
\begin{eqnarray}
	&& |\theta_1|^p|U+W|^{p-2}(U+W)U + |\theta_2|^p|U+W|^{p-2}(U+W)W \nonumber \\
	&=& |\theta_1U+\theta_1W|^{p-2}(\theta_1U+\theta_1W)\theta_1U + |\theta_2U+\theta_2W|^{p-2}(\theta_2U+\theta_2W)\theta_2W \nonumber \\
	&\geq& |\theta_1U + \theta_2W|^p. \nonumber
\end{eqnarray}
Hence,
$$
I(f(\theta)) \leq \lambda, \forall \theta \in \mathbb{S}^1.
$$
By appealing to the definition of $\lambda_{2}(\Omega)$ we get the desired conclusion. Noticing that $\lambda_{1}(p,V) < \lambda_{2}(\Omega)$, we know that $\lambda_{1}(p,V)$ is isolated.
\end{proof}

\begin{cor}
	$\lambda_{2}(\Omega) = \lambda_{2}$ where $\lambda_{2}$ is defined by $\mathbb{Z}_2$-cohomological index.
\end{cor}

\begin{proof}
Obviously, $\lambda_{2}(\Omega) \leq \lambda_{2}$ since $\lambda_{2} > \lambda_{1}(p,V)$ is an eigenvalue. Next we prove that $\lambda_{2}(\Omega) \geq \lambda_{2}$

For any $f \in \mathcal{C}_{1}(\Omega)$, we define $M = \{f(\theta): \theta \in \mathbb{S}^1\}$, which is a symmetric set. It is not difficult to verify that $i(M) \geq 2$. Thus
$$
\sup_{\theta \in \mathbb{S}^1}I(f(\theta)) = \sup_{u \in M}I(u) \geq \inf_{A \in \mathcal{F}_2}\sup_{u\in A}I(u) = \lambda_2, \forall f \in \mathcal{C}_{1}(\Omega).
$$
Hence,
$$
\lambda_{2}(\Omega) = \inf_{f \in \mathcal{C}_{1}(\Omega)}\sup_{\theta \in \mathbb{S}^1}I(f(\theta)) \geq \lambda_2,
$$
and we complete the proof.
\end{proof}

\section{The Case $\mu>0$ versus a discontinuity spectrum}

\begin{prop}
  For $\mu>0$, we set
\begin{equation}\label{lambda-}
  \lambda\left(p, q, \mu, V\right):=\inf \left\{\frac{\mathcal{I}(u)}{\int_{\Omega} V(x)|u|^{p} d x} : u \in \mathcal{W}^s, \int_{\Omega} V(x)|u|^{p} d x>0\right\},
\end{equation}
where
$$
\mathcal{I}(u)=\mathcal{J}_s(u)+\frac{1}{p}\|u\|_p^{p}+\frac{\mu}{q}\|u\|_q^{q} .
$$

 Then,
$$
\lambda\left(p, q, \mu, V\right)=\lambda_{1}\left(p,V\right)
$$
holds for every $\mu>0 .$ In addition, for every $\mu>0$, the infimum in \eqref{lambda-} is not attained.
\end{prop}

\begin{proof}
Fix $\mu>0$. From the definitions of $\lambda_{1}\left(p, V\right)$ and $\lambda\left(p, q, \mu, V\right)$, we can see that $\lambda\left(p, q, \mu, V\right) \geq \lambda_{1}\left(p, V\right)$. Let $u_1$ be the positive eigenfunction corresponding to $\lambda_{1}\left(p, V\right)$ such that $\|u_1\|_{p,V}=1$. Thus, $u_{1}$ satisfies $\|u_1\|_{s,p}^{p}=\lambda_{1}\left(p, V\right)$. 

By using the homogeneity of the norms and letting $t \rightarrow \infty$ in the following inequality

$$
\lambda\left(p, q, \mu, V\right) \leq \frac{\|tu_1\|_{s,p}^p+ \frac{\mu }{q}\|tu_1\|_{s,q}^q}{\|tu_1\|_{p,V}^p}
=
\lambda_{1}(p, V)+\frac{\mu }{q}\frac{\|u_1\|_{s,q}^q}{t^{p-q}}.
$$
 We obtain $\lambda\left(p, q, \mu, V\right) \leq \lambda_{1}\left(p, V\right)$ since $p>q$.

We now prove that $\lambda\left(p, q, \mu, V\right)$ is not attained for any $\mu>0$.  Assume that there exist $\mu>0$ and a function $u \in \mathcal{W}^s$ such that $\|u\|_{p,V}>0$ and
 $$
 \lambda\left(p, q, \mu, V\right)=\frac{\|u\|_{s,p}^p + \frac{\mu }{q}\|u\|_{s,q}^q}{\|u\|_{p,V}^p}
 $$
Then, since $u \neq 0$, the definition of $\lambda_{1}\left(p, V\right)$ (note
that $\mathcal{W}^s \subset W_{0}^{s, q}(\Omega)$ and the first assertion of the proposition, that
\begin{align*}
\lambda\left(p, q, \mu, V\right)=\frac{\|u\|_{s,p}^p + \frac{\mu }{q}\|u\|_{s,q}^q}{\|u\|_{p,V}^p}
\geq \lambda_1(p,V)+
\frac{\mu }{q} \frac{\|u\|_{s,q}^q}{\|u\|_{p,V}^p}
&>\lambda_{1}\left(p, V\right)=\lambda\left(p, q, \mu, V\right).
\end{align*}
This is a contradiction.
\end{proof}

\begin{proof}[Proof of Theorem \ref{nonexistence}]
$(i)$ Fix $\mu>0 .$ Let $u$ be a non-trivial solution of \eqref{mainequation}. By taking $u$ as test function in \eqref{mainequation} and since $\mu>0$ we have
$$
0<[u]_{s,p}^p<[u]_{s,p}^p+\mu[u]_{s,q}^q=\lambda \int_{\Omega} V|u|^{p} d x.
$$
 This yields that one of the following alternatives $(a)$ or $(b)$ occurs:
 \begin{enumerate}
   \item [(a)] $\lambda<0$ and $\int_{\Omega} V|u|^{p} d x<0$ or
   \item [(b)] $\lambda>0$ and $\int_{\Omega} V|u|^{p} d x>0$.
 \end{enumerate}

It is sufficient to only consider the case $(b)$ because we may consider the pair of $-\lambda$ and $-V$ instead of $\lambda$ and $V$ provided $V$ changes sign. Thus, we may assume that $\int_{\Omega} V|u|^{p} d x>0$ and $\lambda>0$, whence we have
$$
\lambda_{1}\left(p, V\right) \leq \frac{[u]_{s,p}^p}{\|u\|_{V,p}^p}<
\frac{[u]_{s,p}^p + \mu[u]_{s,q}^q}{\|u\|_{p,V}^p}=
\lambda
$$
by the definition of $\lambda_{1}\left(p, V\right)$. The first assertion follows.

$(ii)$ Arguing by contradiction, we assume that \eqref{mainequation} has a sign-changing solution $u \in \mathcal{W}^s$ for some $\mu>0$ and $\lambda$ satisfying $\lambda_{1}\left(p, V\right)<\lambda \leq \lambda_{2}\left(p, V\right)$. Then, taking $\pm u_{\pm}$as test functions, we obtain
$$
\begin{aligned}
&0<[ u_{+}]_{s,p}^{p}<[ u_{+}]_{s,p}^{p}+\mu[ u_{+}]_{s,q}^{q}\leq\lambda \int_{\Omega} V u_{+}^{p} d x \quad \text { and } \\
&0<[ u_{-}]_{s,p}^{p}<[ u_{-}]_{s,p}^{p}+\mu[ u_{-}]_{s,q}^{q}\leq\lambda \int_{\Omega} V u_{-}^{p} d x
\end{aligned}
$$
where $u_{\pm}:=\max \{\pm u, 0\} .$ Note that $u_{\pm} \in \mathcal{W}^s $ and $J\left(u_{\pm}\right)=\int_{\Omega} V u_{\pm}^{p} d x>0$ because
$\lambda>0$. Combining these inequalities and the argument in \cite[Proposition 11]{Tanaka}, we can construct a continuous path $\gamma_{0} \in \mathcal{C}_{1}(\Omega)$ such that
$$
\max _{t \in \mathbb{S}^{1}} \mathcal{I}\left(\gamma_{0}(t)\right)<\lambda.
$$

This leads to $\lambda_2(p,V)<\lambda$, which contradicts $\lambda\leq \lambda_2(p,V)$.

\end{proof}

The proof of assertion $(iii)$ of Theorem \ref{nonexistence} will be divided into
several lemmas. Let first consider the Nehari manifold
$$
\begin{aligned}
\mathcal{N}_{\lambda} &:=\left\{u \in \mathcal{W}^s \backslash\{0\}:\left\langle \mathcal{J}^{\prime}(u), u\right\rangle=0\right\} \\
&=\left\{u \in \mathcal{W}^s \backslash\{0\}:[u]_{s, p}^{p}+\mu[u]_{s, q}^{q}+\int_{\Omega} \left(|u|^{p}+\mu|u|^{q}\right) d x=\lambda \int_{\Omega}V |u|^{p} d x\right\}
\end{aligned}
$$
Note that for each $u \in \mathcal{N}_{\lambda}$ the functional $\mathcal{J}$ has the following expression
\begin{equation}\label{Npositive}
  \mathcal{J}(u)=\mu\left(\frac{1}{q}-\frac{1}{p}\right)\left([u]_{s, q}^{q}+\|u\|_q^q\right)\geq0.
\end{equation}

\begin{lem}
  $\mathcal{N}_{\lambda} \neq \emptyset$.
\end{lem}
\begin{proof}
  Since $\lambda>\lambda_{1}(p,V)$, then there exists $w \in \mathcal{W}^s \backslash\{0\}$ for which
$$
\lambda \int_{\Omega}V |\omega|^{p} d x>[\omega]_{s,p}^{p}+\|\omega\|_p^p.
$$
Consider the function $\xi\colon (0,+\infty)\to \R$ defined by
\begin{align}\label{xi}
\xi(t)&=\left\langle \mathcal{J}^{\prime}(t\omega), t\omega\right\rangle\nonumber\\
   &=t^p [\omega]_{s, p}^{p}+\mu t^q[\omega]_{s, q}^{q}+t^p\int_{\Omega} |\omega|^{p}d x+\mu t^q \int_{\Omega}|\omega|^{q} d x-\lambda t^p\int_{\Omega}V |\omega|^{p} d x\nonumber\\
   &=t^q\mu\left([\omega]_{s, q}^{q}+ \int_{\Omega}|\omega|^{q} d x\right)-t^p\left(\int_{\Omega}V |\omega|^{p} d x-[\omega]_{s, p}^{q}- \int_{\Omega}|\omega|^{p} d x\right)
\end{align}
Since $p>q$, we see from \eqref{xi} that
$$
\xi(t)\to -\infty\quad\text{as}\ t\to+\infty.
$$
On the other hand, for small $t\in (0,1)$ we have
$$
\xi(t)>0.
$$
Therefore, we can find $t_0 > 0$ such that $\xi(t_0)=0$ then $t_0 \omega\in \mathcal{N}_{\lambda}$ and so $\mathcal{N}_{\lambda}\neq \emptyset$.
\end{proof}

We define
\begin{equation}\label{mlambda}
  m_\lambda=\inf\{\mathcal{J}(u)\colon u\in \mathcal{N}_{\lambda}\}.
\end{equation}
By \eqref{Npositive} we deduce that $m_\lambda\geq0$ for all $u\in\mathcal{N}_{\lambda}$.

\begin{lem}\label{lemexis1}
$~$
  \begin{enumerate}
    \item [(i)] Every minimizing sequence of \eqref{mlambda} is bounded in $\mathcal{W}^s$.
    \item [(ii)] $m_\lambda>0$.
  \end{enumerate}
\end{lem}

\begin{proof}
$(i)$  We argue by contradiction. So, suppose that $(u_n)\subset \mathcal{W}^s$ is a minimizing sequence of \eqref{mlambda} such that
$$
[u_n]_{s}\to +\infty.
$$
Without loss of generality we can consider $$[u_n]_{s,p}\to +\infty.$$
We have
\begin{equation}\label{bound2}
  [u_n]_{s, p}^{p}+\mu[u_n]_{s, q}^{q}+\int_{\Omega} \left(|u_n|^{p}+\mu|u_n|^{q}\right) d x=\lambda \int_{\Omega}V |u_n|^{p} d x,
\end{equation}
and
\begin{equation}\label{bound1}
  [u_n]_{s, p}^{p}\leq \lambda\|u_n\|_{V,p}^p\quad\Rightarrow\quad \|u_n\|_{V,p}\to+\infty.
\end{equation}

We set $\omega_n=\frac{u_n}{\|u_n\|_{V,p}}$ for all $n\in\mathbb{N}$, hence $\|\omega_n\|_{V,p}=1$. Also, from \eqref{bound1} we have
$$
[\omega_n]_{s, p}^{p}\leq\lambda\quad\text{for all}\ n\in\mathbb{N},
$$
so $(\omega_n)\subset \mathcal{W}^s$ is bounded and we may assume that
$$
\omega_n\rightharpoonup \omega\ \text{in}\ \mathcal{W}^s\quad\text{and}\quad\omega_n\to\omega\ \text{in}\ L^p(V,\Omega).
$$
Recall that $(u_n)\subset \mathcal{W}^s$ is a minimizing sequence for \eqref{mlambda}. So, we have
\begin{equation}\label{bound3}
  \left([u_n]_{s, q}^{q}+\|u_n\|_q^q\right)\to \frac{1}{\mu}\left(\frac{1}{q}-\frac{1}{p}\right)^{-1}m_\lambda\ \text{as}\ n\to\infty.
\end{equation}
We multiply \eqref{bound2} with $1/\|u_n\|_{V,p}$, we obtain
\begin{equation}\label{bound4}
  \mu\left([\omega_n]_{s, q}^{q}+\|\omega_n\|_q^{q}\right)=\lambda-[\omega_n]_{s, p}^{p}-\|\omega_n\|_p^{p}.
\end{equation}
Using \eqref{bound3} in \eqref{bound4}, we can infer that $\omega=0$. On the other hand, since $\|\omega_n\|_{V,p}=1$ for each $n$ we get that $\|\omega\|_{V,p}=1$ which is a contradiction with $\omega=0$. Therefore we can conclude that every minimizing sequence of \eqref{mlambda} is bounded in $\mathcal{W}^s$.

$(ii)$ We already observed that $m_\lambda\geq0$. Suppose by contradiction that $m_\lambda= 0$. Let $(u_n)\subset \mathcal{W}^s$ is a minimizing sequence of \eqref{mlambda}. By similar arguments as the previous proof we show in one hand that $\omega_n\to 0$ in $L^p(V,\Omega)$ and on other hand $\omega_n=\frac{u_n}{\|u_n\|^p_{V,p}}\to 1$ which gives a contradiction and we have the result.

\end{proof}

\begin{lem}
  There exists $u\in \mathcal{N}_{\lambda}$ such that $\mathcal{J}(u)=m_\lambda$.
\end{lem}

\begin{proof}
Let $\left\{u_{n}\right\}_{n \geqslant 1} \subseteq \mathcal{N}_{\lambda}$ be such that $\mathcal{J}\left(u_{n}\right) \rightarrow m_{\lambda}$. According to Lemma \ref{lemexis1}, $\left\{u_{n}\right\}_{n \geqslant 1} \subseteq \mathcal{W}^s$ is bounded. So, we may assume that

\begin{equation}\label{lemex1}
  u_{n} \rightharpoonup \hat{u}_{\lambda} \quad \text { in } \mathcal{W}^s \quad \text { and } \quad u_{n} \rightarrow \hat{u}_{\lambda} \quad \text { in } L^{p}(V,\Omega).
\end{equation}

Since $u_{n} \in \mathcal{N}_{\lambda}$ for all $n \in \mathbb{N}$, we have

\begin{equation}\label{bound2ex}
  [u_n]_{s, p}^{p}+\mu[u_n]_{s, q}^{q}+\int_{\Omega} \left(|u_n|^{p}+\mu|u_n|^{q}\right) d x=\lambda \int_{\Omega}V |u_n|^{p} d x,
\end{equation}

Passing to the limit as $n \rightarrow \infty$ and using \eqref{lemex1} and the weak lower semicontinuity of the norm functional in a Banach space, we obtain
\begin{equation}\label{lemex2}
  \mu\left([\hat{u}_{\lambda}]_{s, q}^{q}+\|\hat{u}_{\lambda}\|_q^{q}\right) d x\leq\lambda\|\hat{u}_{\lambda}\|_{V,p}-[\hat{u}_{\lambda}]_{s, p}^{p}+\|\hat{u}_{\lambda}\|_p^{p}.
\end{equation}
Note that $\lambda\|\hat{u}_{\lambda}\|_{V,p}-[\hat{u}_{\lambda}]_{s, p}^{p}+\|\hat{u}_{\lambda}\|_p^{p} \neq 0$, or otherwise from \eqref{bound4}, we have
$$[u_n]_{s, q}^{q}+\|u_n\|_{q}^{q} \rightarrow 0, $$
so $u_{n} \rightarrow 0$ in $\mathcal{W}^s.$ Recall that
$$
 \mathcal{J}(u_n)=\mu\left(\frac{1}{q}-\frac{1}{p}\right)\left([u_n]_{s, q}^{q}+\|u_n\|_q^q\right) \quad \text { for all } n \in \mathbb{N}.
$$
So, it follows that
$$
 \mathcal{J}(u_n) \rightarrow 0 \quad \text { as } n \rightarrow \infty \quad \text{and}\quad m_{\lambda}=0,
$$
which contradicts Lemma \ref{lemexis1}. Therefore
$$
\lambda\|\hat{u}_{\lambda}\|_{V,p}-[\hat{u}_{\lambda}]_{s, p}^{p}+\|\hat{u}_{\lambda}\|_p^{p} \neq 0\quad \text{and}\quad\hat{u}_{\lambda} \neq 0.
$$
Also, exploiting the sequential weak lower semicontinuity of $\mathcal{J}$, we have
$$
\mathcal{J}\left(\hat{u}_{\lambda}\right) \leqslant \lim _{n \rightarrow \infty} \mathcal{J}\left(u_{n}\right)=m_{\lambda}
$$
If we show that $\hat{u}_{\lambda} \in \mathcal{N}_{\lambda}$, then $\mathcal{J}\left(\hat{u}_{\lambda}\right)=m_{\lambda}$, and this will conclude the proof. To this end, let
$$
\hat{\xi}_{\lambda}(t)=\left\langle\mathcal{J}^{\prime}\left(t \hat{u}_{\lambda}, t \hat{u}_{\lambda}\right)\right\rangle \quad \text { for all } t \in[0,1].
$$
Evidently, $\hat{\xi}_{\lambda}(\cdot)$ is a continuous function. Arguing by contradiction, suppose that $\hat{u}_{\lambda} \notin \mathcal{N}_{\lambda}$. Then since $u_{n} \in \mathcal{N}_{\lambda}$ for all $n \in \mathbb{N}$, we can infer from \eqref{lemex2} that
\begin{equation}\label{lemex3}
  \hat{\xi}_{\lambda}(1)<0.
\end{equation}

On the other hand, note that since $\lambda>\lambda_{1}(p,V)$ and from \eqref{xi}, we have
$$
\hat{\xi}_{\lambda}(t) \geqslant c_{1} t^{q}-c_{2} t^{p} \quad \text { for some } c_{1}, c_{2}>0.
$$
So, since $q<p$, we have
\begin{equation}\label{lemex4}
  \hat{\xi}_{\lambda}(t)>0\quad\text{ for all}\quad t \in(0, \epsilon)\quad\text{ with small}\ \epsilon \in(0,1).
\end{equation}

By \eqref{lemex3} and \eqref{lemex4}, we see that there exists $t^{*} \in(0,1)$ such that
$$
 \hat{\xi}_{\lambda}\left(t^{*} \hat{u}_{\lambda}\right)=0 \quad \text{and}\quad t^{*} \hat{u}_{\lambda} \in \mathcal{N}_{\lambda}.
$$
Then using \eqref{Npositive}, we have
$$
\begin{aligned}
m_{\lambda} \leqslant \varphi_{\lambda}\left(t^{*} \hat{u}_{\lambda}\right) &=\alpha\left(\frac{1}{q}-\frac{1}{p}\right)\left(t^{*}\right)^{q}[\hat{u}_{\lambda}]_{s, q}^{q}+\|\hat{u}_{\lambda}\|_q^{q} \\
&<\alpha\left(\frac{1}{q}-\frac{1}{p}\right)[\hat{u}_{\lambda}]_{s, q}^{q}+\|\hat{u}_{\lambda}\|_q^{q} \quad \left(\text {since } t^{*} \in(0,1)\right) \\
&\leqslant \alpha\left(\frac{1}{q}-\frac{1}{p}\right) \liminf _{n \rightarrow \infty}\left([u_n]_{s, q}^{q}+\|u_n\|_q^{q}\right)\\
&=m_{\lambda}
\end{aligned}
$$
a contradiction. Therefore $\hat{u}_{\lambda} \in \mathcal{N}_{\lambda}$, and this finishes the proof.
\end{proof}

\section{Proofs of Stability results}

First, we give a convergence principle of solutions of a more general type equation with this form

\begin{equation}\label{cvps-equation}
   \left\{\begin{aligned}
(- \Delta)_p^{s}\, u + \mu(- \Delta)_q^{s}\, u+ |u|^{p-2}u+\mu|u|^{q-2}u=f(x,u)\quad & \text{in } \Omega\\
u=0\quad& \text{in}\quad\R^N\backslash\Omega,
\end{aligned}\right.
\end{equation}

Here,  we need to impose some conditions on $f$ in order to guarantee that weak solutions are well defined:
\begin{enumerate}
  \item [($f_1$)] $ f: \Omega \times \mathbb{R} \rightarrow \mathbb{R}$ is a Carath\'eodory function, i.e. $f(\cdot, z)$ is measurable for any $z \in \mathbb{R}$ and $f(x, \cdot)$ is continuous a.e. $x \in \Omega$.
  \item [($f_2$)] There exist a constant $C>0$ such that $|f(x, z)| \leq C(1+|z|)^{r-1}$ for some $r \in\left[1, p_s^{*}\right)$.
\end{enumerate}
We will prove that any accumulation point of the sequence of solutions is in fact a solution to the local limit problem

\begin{equation}\label{cvp-equation}
   \left\{\begin{aligned}
(- \Delta)_p\, u + \mu(- \Delta)_q\, u+ |u|^{p-2}u+\mu|u|^{q-2}u=f(x,u)\quad & \text{in } \Omega\\
u=0\quad& \text{in}\quad\partial\Omega.
\end{aligned}\right.
\end{equation}

\begin{thm}\label{cvp}
  Let $0<s_{k} \rightarrow 1$ and let $u_{k} \in \mathcal{W}^{s_k}$ be a sequence of solutions of \eqref{cvps-equation} such that $\sup _{k \in \mathbb{N}}\left[u_{k}\right]_{s_{k}}<\infty .$ Then, any accumulation point $u$ of the sequence $\left\{u_{k}\right\}_{k \in \mathbb{N}}$ in $L^{p}(\Omega)$ verifies that $u \in W_{0}^{1, p}(\Omega)$ and it is a weak solution of \eqref{cvp-equation}.
\end{thm}

\begin{proof}
 Assume that $u_{k} \rightarrow u$ in $L^{p}(\Omega)$. Then, since $\left\{u_{k}\right\}_{k \in \mathbb{N}}$ is uniformly bounded in $\mathcal{W}^{s_k}$, by \cite{Brezis} we obtain that $u \in W_{0}^{1, p}(\Omega)$. Passing to a subsequence, if necessary, we can also assume that $u_{k} \rightarrow u$ a.e. in $\Omega$.

On the other hand, if we define
$$
\eta_{1,k}:=\left(-\Delta_{p}\right)^{s_{k}} u_{k} \in \left(\mathcal{W}^{s_k}\right)^{\prime}\subset W^{-1, p^{\prime}}(\Omega)
$$
and
$$
\eta_{2,k}:=\left(-\Delta_{q}\right)^{s_{k}} u_{k} \in W^{-s_{k}, q^{\prime}}(\Omega) \subset W^{-1, p^{\prime}}(\Omega)\subset W^{-1, q^{\prime}}(\Omega),
$$
 then $\left\{\eta_{1,k}\right\}_{k \in \mathbb{N}}$ and $\left\{\eta_{2,k}\right\}_{k \in \mathbb{N}}$ are bounded in $W^{-1, p^{\prime}}(\Omega)$ and $W^{-1, q^{\prime}}(\Omega)$ respectively, and hence, up to a subsequence, there exists $\eta \in W^{-1, p^{\prime}}(\Omega)$ such that
  $$
  \eta_{k}:=\eta_{1,k}+\eta_{2,k} \rightarrow \eta\quad\text{ weakly in}\quad W^{-1, p^{\prime}}(\Omega).
  $$
Since $u_{k}$ solves \eqref{cvps-equation}, for any $v \in C_{c}^{\infty}(\Omega)$
$$
0=\left\langle\left(-\Delta_{p}\right)^{s_{k}} u_{k}+\mu\left(-\Delta_{q}\right)^{s_{k}} u_{k}, v\right\rangle-\int_{\Omega} f\left(x, u_{k}\right) v d x+\int_{\Omega}(|u_{k}|^{p-2}u_{k}+\mu|u_{k}|^{q-2}u_{k})v d x.
$$
Using the convergence, taking the limit $k \rightarrow \infty$ we get
$$
0=\langle\eta, v\rangle-\int_{\Omega} f(x, u) v d x+\int_{\Omega}(|u|^{p-2}u+\mu|u|^{q-2}u)v d x.
$$

We want to identify $\eta$, more precisely, we will prove that
\begin{equation}\label{ineq5.1}
  \langle\eta, v\rangle=\left\langle-\Delta_{p} u-\mu\Delta_{p} u, v\right\rangle.
\end{equation}

For that purpose we use the monotonicity of the operator and the fact that $u_{k}$ is solution of \eqref{cvps-equation}, Indeed,
$$
\begin{aligned}
0 & \leq\left\langle\left(-\Delta_{p}\right)^{s_{k}} u_{k}+\mu\left(-\Delta_{q}\right)^{s_{k}} u_{k}, u_{k}-v\right\rangle-\left\langle\left(-\Delta_{p}\right)^{s_{k}} v+\mu \left(-\Delta_{q}\right)^{s_{k}} v, u_{k}-v\right\rangle \\
&=\int_{\Omega} f\left(x, u_{k}\right)\left(u_{k}-v\right) d x-\left\langle\left(-\Delta_{p}\right)^{s_{k}} v+\mu \left(-\Delta_{q}\right)^{s_{k}} v, u_{k}-v\right\rangle.
\end{aligned}
$$
Hence taking the limit $k \rightarrow \infty$ and using Lemma \ref{cvps} one finds that
$$
\begin{aligned}
0 & \leq \int_{\Omega} f(x, u)(u-v)-\left\langle-\Delta_{p} v-\mu\Delta_{q} v, u-v\right\rangle \\
&=\langle\eta, u-v\rangle-\left\langle-\Delta_{p} v-\Delta_{q} v, u-v\right\rangle.
\end{aligned}
$$
Consequently, if we take $v=u-t w, w \in W_{0}^{1, p}(\Omega)$ given and $t>0$, we obtain that
$$
0 \leq\langle\eta, w\rangle-\left\langle-\Delta_{p}(u-t w)-\mu\Delta_{q}(u-t w), w\right\rangle
$$
taking $t \rightarrow 0^{+}$ gives that
$$
0 \leq\langle\eta, w\rangle-\left\langle-\Delta_{p}(u)-\mu\Delta_{q}(u), w\right\rangle.
$$
Thus \eqref{ineq5.1} deduced and this completes the proof.
\end{proof}

Focusing on the fractional parameter we can redefine our equation \eqref{mainequation} as follows

\begin{equation}\label{mainequationlambdas}
  \left\{\begin{aligned}
(- \Delta)_p^{s}\, u + \mu(- \Delta)_q^{s}\, u+ |u|^{p-2}u+\mu|u|^{q-2}u=\lambda^s\ V(x)|u|^{p-2}u\quad & \text{in } \Omega\\
u=0\quad& \text{in}\quad\R^N\backslash\Omega,
\end{aligned}\right.
\end{equation}
and its local counterpart
\begin{equation}\label{mainequationlambda}
 \left\{\begin{aligned}
(- \Delta)_p\, u + \mu(- \Delta)_q\, u+ |u|^{p-2}u+\mu|u|^{q-2}u=\lambda^1\ V(x)|u|^{p-2}u\quad & \text{in } \Omega\\
u=0\quad& \text{in}\quad\R^N\backslash\Omega.
\end{aligned}\right.
\end{equation}

\begin{thm}\label{thm5.2}
Let $\lambda^{s} \in \sigma_{p,q}^{s}$ be such that $\sup _{s \in(0,1)} \lambda^{s}<\infty .$ Then any accumulation point of the set $\{\lambda^s\}_{s\in (0,1)}$ belongs to $\sigma^1_{p,q}.$
 Moreover, if $\left\{s_{k}\right\}_{k \in \mathbb{N}}$ is such that $s_{k} \rightarrow 1$ and $\lambda^{s_{k}} \rightarrow \lambda$ and $u_{k} \in \mathcal{W}^{s_k}$ is an eigenfunction  of \eqref{mainequationlambdas} associated to $\lambda^{s_{k}}$, then, up to a further subsequence, there exists $u \in W_{0}^{1, p}(\Omega)$ such
that $u_{k} \rightarrow$ u strongly in $L^{p}(\Omega)$ and $u$ is an eigenfunction of \eqref{mainequationlambda} associated to $\lambda$.
\end{thm}

\begin{proof}
Assume that $\lambda^{k}:=\lambda^{s_{k}} \rightarrow \lambda>\lambda_1(p,V)$ and let $u_{k}$ be the associated eigenfunction of \eqref{mainequationlambdas}. First, let us prove that
\begin{equation}\label{boundness}
  \left\|u_{k}\right\|_{s_{k}, p}^{p}\leq\|u_k\|_{s_k} \leq M,
\end{equation}
with $M$ independent on $k \in \mathbb{N}$ and $\|\cdot\|_s:=\|\cdot\|_{s,p}^p+\|\cdot\|_{s,q}^q.$

Since $\lambda^k \searrow\lambda>\lambda_{1}(p,V)$, then for some $s\in (0,1)$ and $0<\varepsilon<\lambda,$ there exists $\omega \in \mathcal{W}^s \backslash\{0\}$ (independent of $k$) for which
\begin{equation}\label{star}
  \|\omega\|_{s,p}^p<(\lambda-\varepsilon)\|\omega\|_{p,V}^p<\lambda^k\|\omega\|_{p,V}^p.
\end{equation}
Consider the function $\Phi\colon (0,+\infty)\to \R$ defined by
$$
\begin{aligned}
\Phi(t)=\mathcal{J}(t\omega)&=t^p\frac{\|\omega\|_{s,p}^p}{p}+t^q\frac{\|\omega\|_{s,q}^q}{q}-t^p \lambda^k\frac{\|\omega\|_{p,V}^p}{p}\\
&:=\alpha t^p+\beta t^q-\gamma_k t^p=(\alpha-\gamma_k)t^p+\beta t^q.
\end{aligned}
$$
Passing to the derivative:
$$
\Phi^\prime(t)=0\quad\Leftrightarrow\quad t_k=\left(\frac{\beta q}{p(\alpha-\gamma_k)}\right)^{\frac{1}{p-q}}.
$$
From \eqref{star}, we get that $t_k\to t_0$ for some $t_0>0$, then $\Phi^\prime(t_0)=0$ and
\begin{equation}\label{e0}
  e_0:=t_0\omega\in \mathcal{N}_{\lambda_k}\quad\text{for some}\ k_0,\ \text{and all}\ k>k_0.
\end{equation}
Note that for each $u_k \in \mathcal{N}_{\lambda_k}$ the functional $\mathcal{J}$ has the following expression
\begin{equation}\label{Npositivek}
  \mathcal{J}(u_k)=\mu\left(\frac{1}{q}-\frac{1}{p}\right)\|u_k\|_{s,q}^q\geq0.
\end{equation}
On the other hand, from \eqref{e0} we have $\mathcal{J}(u_k)=m_{\lambda_k}\leq \mathcal{J}(e_0).$ So
\begin{equation}\label{Jukbounded}
  |\mathcal{J}(u_k)|\leq c,\quad \text{for some}\ c>0.
\end{equation}

Let $\theta>p>q$. We can see that
$$
\begin{aligned}
  C_0(1+\|u_k\|_{s_k}) & \geq \mathcal{J}(u_k)+\frac{1}{\theta}\left\langle \mathcal{J}^{\prime}(u_k), u_k\right\rangle\\
  &= \left(\frac{1}{p}-\frac{1}{\theta}\right)\|u_k\|_{s,p}^p+\left(\frac{1}{q}-\frac{1}{\theta}\right)\|u_k\|_{s,q}^q
  -\left(\frac{1}{p}-\frac{1}{\theta}\right)\lambda^k\|u_k\|_{p,V}^p\\
  &\geq \left(\frac{1}{p}-\frac{1}{\theta}\right)\left(\|u_k\|_{s,p}^p+\|u_k\|_{s,q}^q\right)=:C_1\left(\|u_k\|_{s,p}^p+\|u_k\|_{s,q}^q\right).
\end{aligned}
$$
Now, assume by contradiction that $\|u_k\|_{s_k}\to\infty.$ Then from \eqref{Npositivek} and \eqref{Jukbounded}, we have $\|u_k\|_{s_k,p}^p\to \infty$ and $\|u_k\|_{s_k,q}^q$ is bounded. Therefore,
$$
C_0\left(\|u_k\|_{s_k,p}+\|u_k\|_{s_k,q}\right)= C_0(1+\|u_k\|_{s_k})\geq C_2 \|u_k\|_{s_k,p}^p,
$$
and thus
$$
C_0\left(\frac{1}{\|u_k\|_{s,p}^p}+\frac{1}{\|u_k\|_{s,p}^{p-1}}+\frac{1}{\|u_k\|_{s,p}^{p-q}}\right)\geq C_2.
$$
Since $p>q>1$ and passing to the limit as $k\to\infty$ we deduce that $C_2\leq 0$ which is a contradiction.

Therefore, \eqref{boundness} is proved and by \cite{Brezis}, there exists $u \in W_{0}^{1, p}(\Omega)$ such that $u_{k} \rightarrow u$ strongly in $L^{p}(\Omega)$.

 We claim that $u\neq0,$ otherwise, we have $\left\|u_{k}\right\|_{s_{k}, p}\to 0,$ and $\left\|u_{k}\right\|_{V,p}^{p}\to 0.$ By arguing as the proof of Lemma \ref{lemexis1} $(i)$, we get that
$$
\omega_n\rightharpoonup \omega\ \text{in}\ \mathcal{W}^s\quad\text{and}\quad\omega_n\to\omega\ \text{in}\ L^p(V,\Omega),
$$
with $\omega\neq0.$ Fix $\varphi\in C^{\infty}_0$ such that
$$
\left\langle\left(-\Delta\right)_{q}\omega,\varphi \right\rangle+\int_{\Omega}|\omega|^{q-2}\omega \varphi dx>0.
$$
Then, by \eqref{mainequationlambdas}, we have
$$
\left\langle\left(-\Delta_p\right)^{s_k}\omega_k,\varphi \right\rangle + \int_{\Omega} |\omega_k|^{p-2} \omega_k \varphi\, dx +\frac{\mu}{\left\|u_{k}\right\|_{V,p}^{p-q}}\left(\left\langle\left(-\Delta_q\right)^{s_k}\omega_k,\varphi \right\rangle+\int_{\Omega} |\omega_k|^{q-2} \omega_k \varphi\, d x\right)=\lambda^{s_k}.
$$
The left-hand side of the previous equality tends towards infinity on the other hand the term on the right is bounded, which is a contradiction, so we have $u\neq0.$

Now, since $f_{k}(z):=\lambda^{k}V |z|^{p-2} z \rightarrow f(z):=\lambda V|z|^{p-2} z$ uniformly on compact sets of $z \in \mathbb{R}$, from Theorem \ref{cvp} it follows that $u$ is an eigenfunction of \eqref{mainequationlambda} associated to $\lambda$ as we wanted to show.
\end{proof}

\end{document}